\documentclass[12pt]{article}

\usepackage{verbatim}
\usepackage{setspace}
\usepackage{sectsty}
\subsectionfont{\normalsize}
\makeatletter
\renewcommand\section{\@startsection{section}{1}{\z@}%
                                  {-2.0ex \@plus -1ex \@minus -.2ex}%
                                  {2.0ex \@plus.2ex}%
                                  {\normalfont\normalsize\bfseries}}
\makeatother
\usepackage{fullpage}
\usepackage{url}
\usepackage{amsmath}
\usepackage{mathdots}
\newcommand{\f}[1]{f_{#1}}

\newcommand{\ang}[1]{\langle#1\rangle}

\newcommand{\nat}{{\sf N}}

\newcommand{\xvec}[1]{\ifcase 3{#1} {\ang {x_1,x_2,x_3} } \else 
\ifcase 4{#1} {\ang{x_1,x_2,x_3,x_4}} \else {\ang {x_1,\ldots,x_{#1}}}\fi\fi}
\newcommand{\yvec}[1]{\ifcase 3{#1} {\ang {y_1,y_2,y_3} } \else 
\ifcase 4{#1} {\ang{y_1,y_2,y_3,y_4}} \else {\ang {y_1,\ldots,y_{#1}}}\fi\fi}
\newcommand{\zvec}[1]{\ifcase 3{#1} {\ang {z_1,z_2,z_3} } \else 
\ifcase 4{#1} {\ang{z_1,z_2,z_3,z_4}} \else {\ang {z_1,\ldots,z_{#1}}}\fi\fi}
\newcommand{\vecc}[2]{\ifcase 3{#2} {\ang { {#1}_1,{#1}_2,{#1}_3 } } \else
\ifcase 4{#1} {\ang { {#1}_1,{#1}_2,{#1}_3,{#1}_{4} } }
\else {\ang { {#1}_1,\ldots,{#1}_{#2}}}\fi\fi}
\newcommand{\veccd}[3]{\ifcase 3{#2} {\ang { {#1}_{{#3}1},{#1}_{{#3}2},{#1}_{{#3}3} } } \else
\ifcase 4{#1} {\ang { {#1}_{{#3}1},{#1}_{{#3}2},{#1}_{#3}3},{#1}_{{#3}4} }
\else {\ang { {#1}_{{#3}1},\ldots,{#1}_{{#3}{#2}}}}\fi\fi}
\newcommand{\veccz}[2]{\ifcase 3{#2} {\ang { {#1}_0,{#1}_2,{#1}_3 } } \else
\ifcase 4{#1} {\ang { {#1}_0,{#1}_2,{#1}_3,{#1}_{4} } }
\else {\ang { {#1}_0,\ldots,{#1}_{#2}}}\fi\fi}
\newcommand{\xve}[1]{\ifcase 3{#1} {x_1,x_2,x_3} \else 
\ifcase 4{#1} {x_1,x_2,x_3,x_4} \else {x_1,\ldots,x_{#1}}\fi\fi}
\newcommand{\yve}[1]{\ifcase 3{#1} {y_1,y_2,y_3} \else 
\ifcase 4{#1} {y_1,y_2,y_3,y_4} \else {y_1,\ldots,y_{#1}}\fi\fi}
\newcommand{\zve}[1]{\ifcase 3{#1} {z_1,z_2,z_3} \else 
\ifcase 4{#1} {z_1,z_2,z_3,z_4} \else {z_1,\ldots,z_{#1}}\fi\fi}
\newcommand{\ve}[2]{\ifcase 3#2 {{#1}_1,{#1}_2,{#1}_3} \else
\ifcase 4#2 {{#1}_1,{#1}_2,{#1}_3,{#1}_{4}}
\else {{#1}_1,\ldots,{#1}_{#2}}\fi\fi}
\newcommand{\ved}[3]{\ifcase 3#2 {{#1}_{{#3}1},{#1}_{{#3}2},{#1}_{{#3}3}} \else
\ifcase 4#2 {{#1}_{{#3}1},{#1}_{{#3}2},{#1}_{{#3}3},{#1}_{{#3}4}}
\else {{#1}_{{#3}1},\ldots,{#1}_{{#3}{#2}}}\fi\fi}
\newcommand{\fuve}[3]{
\ifcase 3#2
{{#3}({#1}_1),{#3}({#1}_2,{#3}({#1}_3)} \else
\ifcase 4#2
{{#3}({#1}_1),{#3}({#1}_2),{#3}({#1}_3),{#3}({#1}_4)}
\else
{{#3}({#1}_1),\ldots,{#3}({#1}_{#2})}\fi\fi}
\newcommand{\setmathchar}[1]{\ifmmode#1\else$#1$\fi}
\newcommand{\vlist}[2]{%
	\setmathchar{%
		\compound#2\one{#2}\two
		\ifcompound
			({#1}_1,\ldots,{#1}_{#2})
		\else
			\ifcat N#2
				({#1}_1,\ldots,{#1}_{#2})
			\else
				\ifcase#2
					({#1}_0)\or
					({#1}_1)\or
					({#1}_1,{#1}_2)\or 
					({#1}_1,{#1}_2,{#1}_3)\or
					({#1}_1,{#1}_2,{#1}_3,{#1}_4)\else 
					({#1}_1,\ldots,{#1}_{#2})
				\fi
			\fi
		\fi}}

\newif\ifcompound
\def\compound#1\one#2\two{%
	\def\one{#1}
	\def\two{#2}
	\if\one\two
		\compoundfalse
	\else
		\compoundtrue
	\fi}

\newcommand{\xwe}[1]{\ifcase 3{#1} {x_1\wedge x_2\wedge x_3} \else 
\ifcase 4{#1} {x_1\wedge x_2\wedge x_3\wedge x_4} \else {x_1\wedge \cdots \wedge
x_{#1}}\fi\fi}
\newcommand{\we}[2]{\ifcase 3#2 {\ang { {#1}_1\wedge {#1}_2\wedge {#1}_3 } } \else
\ifcase 4{#1} {\ang { {#1}_1\wedge {#1}_2\wedge {#1}_3\wedge {#1}_{4} } }
\else {\ang { {#1}_1\wedge \cdots\wedge {#1}_{#2}}}\fi\fi}

\newcommand{\st}{\mathrel{:}}

\newcommand{\es}{\emptyset}

\newcommand{\ceil}[1]{\left\lceil {#1}\right\rceil}

\newcommand{\s}[1]{\s_{#1}}

\newcommand{\monus}{\;\raise.5ex\hbox{{${\buildrel
    \ldotp\over{\hbox to 6pt{\hrulefill}}}$}}\;}

\newcounter{savenumi}

\newtheorem{theoremfoo}{Theorem}[section] 
\newenvironment{theorem}{\pagebreak[1]\begin{theoremfoo}}{\end{theoremfoo}}

\newtheorem{lemmafoo}[theoremfoo]{Lemma}
\newenvironment{lemma}{\pagebreak[1]\begin{lemmafoo}}{\end{lemmafoo}}
\newtheorem{conjecturefoo}[theoremfoo]{Conjecture}

\newtheorem{conventionfoo}[theoremfoo]{Convention}

\newtheorem{porismfoo}[theoremfoo]{Porism}

\newtheorem{gamefoo}[theoremfoo]{Game}

\newtheorem{corollaryfoo}[theoremfoo]{Corollary}

\newtheorem{openfoo}[theoremfoo]{Open Problem}

\newtheorem{exercisefoo}{Exercise}

\newcommand{\fig}[1] 
{
 \begin{figure}
 \begin{center}
 \input{#1}
 \end{center}
 \end{figure}
}

\newtheorem{potanafoo}[theoremfoo]{Potential Analogue}

\newtheorem{notefoo}[theoremfoo]{Note}

\newtheorem{notabenefoo}[theoremfoo]{Nota Bene}

\newtheorem{nttn}[theoremfoo]{Notation}

\newtheorem{empttn}[theoremfoo]{Empirical Note}

\newtheorem{examfoo}[theoremfoo]{Example}

\newtheorem{dfntn}[theoremfoo]{Def}
\newenvironment{definition}{\pagebreak[1]\begin{dfntn}\rm}{\end{dfntn}}

\newtheorem{propositionfoo}[theoremfoo]{Proposition}

\newenvironment{proof}
    {\pagebreak[1]{\narrower\noindent {\bf Proof:\quad\nopagebreak}}}{\QED}

\newcommand{\yyskip}{\penalty-50\vskip 5pt plus 3pt minus 2pt}
\newcommand{\blackslug}{\hbox{\hskip 1pt
        \vrule width 4pt height 8pt depth 1.5pt\hskip 1pt}}
\newcommand{\QED}{{\penalty10000\parindent 0pt\penalty10000
        \hskip 8 pt\nolinebreak\blackslug\hfill\lower 8.5pt\null}
        \par\yyskip\pagebreak[1]}

\newcommand{\BBB}{{\penalty10000\parindent 0pt\penalty10000
        \hskip 8 pt\nolinebreak\hbox{\ }\hfill\lower 8.5pt\null}
        \par\yyskip\pagebreak[1]}

\newtheorem{factfoo}[theoremfoo]{Fact}

\newenvironment{block}{\begin{list}{\hbox{}}{\leftmargin 1em
    \itemindent -1em \topsep 0pt \itemsep 0pt \partopsep 0pt}}{\end{list}}

\begin{document}

\newcommand{\KN}{K_{\nat}}
\newcommand{\NRE}{\hbox{NUM-RED-EDGES\ }}
\newcommand{\NBE}{\hbox{NUM-BLUE-EDGES\ }}
\newcommand{\RED}{\hbox{RED\ }}
\newcommand{\BLUE}{\hbox{BLUE\ }}
\newcommand{\REDns}{\hbox{RED}}
\newcommand{\BLUEns}{\hbox{BLUE}}

\title{Which Unbounded Protocol for Envy Free Cake Cutting is Better?}


\date{}

\maketitle

\begin{abstract}
There are three protocols for envy free cake division for
$n$ people.
All of them take an unbounded number of cuts.
We quantify this unbounded number with ordinals and hence
can say, of the three unbounded algorithms, which one is better.
\end{abstract}

\section{Introduction}

How do $n$ people split a cake fairly?
This is a well studied question, using different definitions of fairness.
See for example the books of Robertson\&Webb~\cite{robertsonwebb}  and
Brams\&Taylor~\cite{BTbook}.
What makes the topic interesting is that 
the people may have different tastes.
Alice likes the left part that has chocolate, where as Bob likes
the right part that has kale.
Formally they all have valuations on the cake.
A valuation is a function with domain well behaved subsets of
the cake and co-domain $[0,1]$. The entire cake maps to 1
and the valuation is additive.

What is fair?
\begin{definition}~
There are $n$ players $A_1,\ldots,A_n$.
\begin{enumerate}
\item
A division is {\it proportional} if everyone gets a piece of cake that they
value as $\ge \frac{1}{n}$.
\item
A division is {\it envy free} if everyone gets a piece of cake that they
think is the best or tied for the best.
\item
A {\it discrete protocol} is a procedure that involves only discrete steps and
ends with a division of the cake.
This is in contrast to moving knife protocols~\cite{BTZsurvey}.
We will use the term {\it $n$-player $c$-cuts proportional protocol} to mean {\it discrete protocol that
results in a proportional division for $n$ players and uses at most $c$ cuts}. 
Similar for {\it envy free protocol}.
\end{enumerate}
\end{definition}

There is a 2-player 1-cut proportional protocol:
Alice cuts, Bob Choose. Even and Paz~\cite{cakecuttingdc}
proved that, for all $n$, there is an $n$-player  $O(n\log n)$ 
proportional protocol.
Edmonds and Pruhs~\cite{cakecuttinglb} proved that $\Omega(n\log n)$ cuts is optimal.

Selfridge and Conway independently obtained a
3-player 5-cut envy free protocol in the early 1960's (unpublished, though in the books
on cake cutting cited above).
It was an open problem for many years to find a 4-player envy free protocol until,
in 1995, Brams and Taylor~\cite{envy free4} 
showed that, for $n\ge 4$ there is an $n$ player envy free protocol.
The protocol uses an unbounded number of cuts. For any particular players
it will use a finite number of cuts; however, that number depends on the players valuations.

Two more unbounded cuts envy free protocols 
have been discovered, due to Robertson and Webb~\cite{ef2} and Pikhurto~\cite{ef3}.
Of the three known protocols,
which is better? We need a way to compare unbounded protocols.

\begin{definition}
Let $\zeta$ be an ordinal. A protocol takes {\it $\zeta$ cuts} if
(1) the protocol starts with $\zeta$ in a counter, 
(2) every time a cut is made an outside observer who knows only what
all of the players know at the time, and wants the protocol to succeed,
decreases the counter, and
(3) at the end of the protocol the counter has a number $\ge 0$.
Note that if (say) $\omega$ is in the counter and a cut is made 
then $\omega$ will
be replaced by some natural number, though it might be large.
The outside observer will know, when the time comes to decrease the
counter, how many cuts are needed to finish the protocol.
\end{definition}

We will sketch variants of the three envy free protocols and analyze them
in terms of how how many cuts they take, both in the worst case and in
the average case. The variants use the essential ideas but try to optimize
the number of cuts. 

In describing protocols we use the convention (from~\cite{envy free4}) that
what a player has to do is described and what the player is advised to do is
written in parenthesis. By {\it should do} we mean that if a player $A$ 
does not follow the advice then $A$ might end up with less than $\frac{1}{n}$.
We do not prove these assertions.

Our results are as follows:
\begin{enumerate}
\item
{\bf The Brams and Taylor Protocol:}
Let $n\in\nat$ and $L=LCM(2,\ldots,n)$.
\begin{enumerate}
\item
The protocol uses $\ceil{\frac{n^2-2n+2}{2}}\omega +L-1$-cuts in the worst case.
\item
The protocol uses $(n+o(1))\omega+L-1$ cuts in the average case (defined suitably).
\end{enumerate}
\item
{\bf The Robertson and Webb Protocol:}
\begin{enumerate}
\item
The protocol uses $(2n-3)\omega$-cuts in the worst case.
\item
The protocol uses $(2n-O(1))\omega$ cuts in the average case (defined suitably).
\end{enumerate}
\item
{\bf The Pikhurto Protocol} is similar to the Robertson and Webb protocol.
\end{enumerate}

\section{The Brams-Taylor Protocol}

\begin{definition}
Let $C$ be a cake to be split and let
$A_1,\ldots,A_n$ be the ones who will split it.
{\it $A_i$ has an advantage over $A_j$} if
$A_i$ does not care how much of $C$ $A_j$ gets.
\end{definition}

\begin{lemma}\label{le:adv}
There is an $n$-player, $\omega$-cuts protocol which will do the following.
The players are $A_1,\ldots,A_n$.
\begin{enumerate}
\item
The input is three pieces $P,Q,R$ such that there are two players 
$A_i,A_j$ where $A_i$ think $P$ and $Q$ are the same size,
but $A_j$ thinks $P$ and $Q$ are different sizes.
\item
At the end of the protocol all but a (small) piece $T$ of $P,Q,R$ are divided
amoung $A_1,\ldots,A_n$.
\item
The division of $P\cup Q\cup R - T$ to $A_1,\ldots,A_n$ is envy free.
\item
$A_i$ and $A_j$ each have an advantage over each other with regard to dividing $T$.
\end{enumerate}
\end{lemma}

\begin{definition}
We call the protocol from Lemma~\ref{le:adv}
{\it the adv $(A_1,\ldots,A_n;A_i,A_j;P,Q,R)$-protocol}.
\end{definition}

\begin{lemma}\label{le:deg}
Let $G$ be a graph on $n$ vertices and $e$ edges.
If $e\ge \ceil{\frac{n(n-2)}{2}+1}=\ceil{\frac{n^2-2n+2}{2}}$ then $G$ must have a vertex of degree $n-1$.
\end{lemma}

\begin{proof}
We prove the contrapositive.
If every vertex is of degree $\le n-2$ then 

$$2e = \sum_{v\in V} d_v  \le n(n-2)$$ 

so $e\le \frac{n(n-2)}{2}$.
\end{proof}

\begin{theorem}\label{th:efbt}
Let $n\in\nat$ and $L=LCM(2,\ldots,n)$.
There is an $n$-person, envy free protocol that has the following properties.
\begin{enumerate}
\item
The protocol uses $\ceil{\frac{n^2-2n+2}{2}}\omega +L-1$-cuts in the worst case.
\item
The protocol uses $(n+o(1))\omega+L-1$ cuts in the average case (defined suitably).
\end{enumerate}
\end{theorem}

\begin{proof}

We give a protocol that has as input 
a cake $C$ and a graph $G$ on $n$ vertices.
If $(A,B)$ is an edge in $G$ then $A$ and $B$ have an advantage over each other
with regard to how $C$ is split.
We will denote the edges of $G$ by $E$.

The protocol is denoted $EFBT(C,G)$ (Envy Free Brams-Taylor). 
It may call itself with a much smaller cake
and a slightly bigger graph.
The players are $A_1,\ldots,A_n$.

\bigskip

\noindent
{\bf PROTOCOL $EFBT(C,G)$.}
\begin{enumerate}
\item
If there is a vertex $A$ of degree $n-1$ in $G$ then nobody else cares
if $A$ gets more cake then they do. So give all of the cake to $A$ and the protocol ends.
Otherwise proceed.
\item
$A_1$ divides the cake into $L=LCM(2,3,\ldots,n)$ pieces. (Equally.)
She uses $L-1$ cuts.
\item
Everyone writes down either $EQ$ or $NEQ$.
$A_1$ has to write down $EQ$.
($A_i$ writes $EQ$ if $A_i$ thinks that all of the pieces are equal,
$NEQ$ if $A_i$ thinks that two of the pieces are not equal.)
\item
What everyone wrote is revealed.
Partition the people into two groups $EQ$ and $NEQ$ based on what they wrote.
\item
Form the bipartite graph $H=(EQ,NEQ,E\cap (EQ\times NEQ))$.
Note that since $A_1\in EQ$, $EQ\ne\es$.
\item

\noindent
{\bf Case 1:}
$H$ is a complete bipartite graph (this includes the case where $NEQ=\es$).
Let $k$ be the number of people in $EQ$.
Nobody in $NEQ$ cares what
anyone in $EQ$ gets.
The people in $EQ$ think all of the pieces are equal.
Each person in $EQ$ gets $L/k$ pieces.
Note that $k$ divides $L$ by the definition of $L$.

\bigskip

\noindent
{\bf Case 2:}
$H$ is not the complete bipartite graph. Let $(i,j)$ be the least pair
lexicographically such that $(A_i,A_j)$ is not an edge. This is not arbitrary:
we would like to use $(A_1,A_j)$ if we can.
Let $P,Q$ be such that $A_i$ thinks $P=Q$ but $A_j$ thinks $P\ne Q$.
Let $R= C- (P\cup Q)$. 
The protocol $adv(A_1,\ldots,A_n;A_i,A_j;P,Q,R)$ is run.
This takes $\omega$ cuts.
Let $C'$ be the cake that is left over.
Call $EFBT(C',G \cup \{i,j\})$.
\end{enumerate}
\noindent
{\bf END OF PROTOCOL $EFBT$}

To envy free divide a cake among $n$ people you would call $EFBT(C,\es)$. 
Once $G$ has a vertex of degree $n-1$ the protocol will stop.
Each iteration adds a pair. By Lemma~\ref{le:deg} 
the number of iterations is bounded by $\ceil{\frac{n^2-2n+2}{2}}$.
Hence the number of cuts is bounded by $\ceil{\frac{n^2-2n+2}{2}}\omega +L-1$.

What happens in the average case? This needs to be defined.
We assume that the 
partitioning of $A_2,\ldots,A_n$ into 
$EQ$ and $NEQ$ is random. Given this, 
we show that the expected number of iterations
before $A_1$ has degree $n-1$ is $n+o(1)$.

Let $E(L)$ be the expected number of iterations before $A_1$ has degree $L$ or
the protocol terminates.
Clearly $E(1)=1$.
If $A_1$ has degree $L-1$ then the probability that in the next iteration $A_1$
will gain a degree or $NEQ=\es$ (so the protocol terminates) is 
$1-\frac{2^{n-(L-1)}}{2^{n-1}}= 1-(0.5)^{L+2}$.
Hence
$E(L)=E(L-1) + \frac{1}{1-(0.5)^{L+2}}$; therefore,
$$E(n)=1+\sum_{i=2}^n\frac{1}{1-(0.5)^{L+1}} \sim \frac{\ln(2^{n+1})-1}{\ln(2)}
= n + o(1).$$
Hence the average case is $(n+o(1))\omega+L-1$ cuts.
\end{proof}

Note that the protocol from Theorem~\ref{th:efbt} yields a 4-person $5\omega$-cuts envy free protocol.

\section{Robertson and Webb Protocol}

In the definitions below we assume that the cake is normalized to have value
1 for everyone. When we use these definitions we may apply them to a piece of cake
that they view differently. We leave it to the reader to make the needed modifications.

\begin{definition}
Let $n,p\in\nat$ and $0\le \epsilon<1$.
A {\it near-exact $(n,p,\epsilon)$ protocol} is one that $n$ people participate in,
and at the end there exists $p$ pieces of cake such that everyone thinks that
every pieces is within
$\epsilon$ of $\frac{1}{p}$.
A {\it near-exact-* $(n,p,\epsilon)$ protocol} is a near exact 
$(n,p,\epsilon)$-protocol where 
one of the players (always $A_1$) thinks all of the pieces are exactly $\frac{1}{p}$.
Note that for near-exact and near-exact-* protocols we do not give cake to anyone.
\end{definition}

The following lemma was first proven by Robertson and Webb~\cite{ef2};
however, Pikhurto~\cite{ef3} later had an especially nice proof.

\begin{lemma}\label{le:ap}
If $n,p\in\nat$ and $0\le \epsilon<1$ then there exists 
a {\it near exact-* $(n,p,\epsilon)$ protocol}.
The number of cuts is a function of $n,p$ and $\epsilon$.
\end{lemma}

\begin{definition}
Let $n\in\nat$, $0<\f 1,\f 2<1$, such that $\f 1+\f 2=1$, $0\le \epsilon<1$.
An {\it unfair near exact $(n,\f 1,\f 2,\epsilon)$ protocol} is one that $n$ people participate in,
and at the end there exists 2 pieces of cake such that everyone thinks that
the first piece is within $\epsilon$ of $\f 1$
and the second piece is within $\epsilon$ of $\f 2$.
(We will not need the $*$-version.)
\end{definition}

\begin{lemma}\label{le:f}
For all $n\in\nat$, $0<\f 1,\f 2<1$, such that $\f 1+\f 2=1$, $0\le \epsilon<1$
there exists an
{\it unfair $(n,\f 1,\f 2,\epsilon)$ protocol}.
The number of cuts depends on $a,f_1,f_2$, and $\epsilon$.
\end{lemma}

\begin{definition}
Let $A_1,\ldots,A_n$ be the people.
A piece of cake $P$ is {\it controversial}
if there exists a nontrivial partition of the people into sets
$S_1$ and $S_2$, and two numbers $\alpha>\beta$ such that
\begin{itemize}
\item
Everyone in $S_1$ thinks that $P$ is worth $\alpha$.
\item
Everyone in $S_2$ thinks that $P$ is worth $\le \beta$.
\end{itemize}
\end{definition}

\begin{definition}
For all $n,m\in\nat$ and $\epsilon>0$ a {\it controversial $(n,m,\delta)$-protocol}
is a protocol for $n+m$ people $A_1,\ldots,A_n$;$B_1,\ldots,B_m$ that starts with a 
piece $P$ that is 
controversial for $A_1,\ldots,A_n$ (we do not know what the $B_i$'s think of $P$),
and ends with a piece $P'$ such that that 
(1) $P'$ is  controversial for $A_1,\ldots,A_n$ (though perhaps
with a different partition than the controversy of $P$), and
(2) everyone (including the $B_i$s) thinks $P'$ is worth $\le \delta$.
\end{definition}

\begin{lemma}\label{le:cont}
For all $n,m\in\nat$ and $\epsilon>0$ there exists an
controversial $(n,m,\epsilon)$-protocol.
The number of cuts depends on $n,m$, and $\epsilon$.
\end{lemma}

\begin{theorem}\label{th:efrw}
Let $m,n\in\nat$ and $\epsilon>0$.
There is a protocol for $n+m$ people $A_1,\ldots,A_n$;$B_1,\ldots,B_m$
that divides a cake $C$ into $n$ pieces, each piece going to one of the $A$-people,
(The $B$-people get nothing!)
such that the following happens.
\begin{enumerate}
\item
The division is envy free for $A_1,\ldots,A_n$.
\item
$A_1,\ldots,A_n,B_1,\ldots,B_m$ all think
that every piece is within
$\epsilon$ of $\frac{1}{n}$.
\end{enumerate}
The number of cuts is as follows.
\begin{enumerate}
\item
The protocol uses $(2n-3)\omega$-cuts in the worst case.
\item
The protocol uses $(2n-O(1))\omega$ cuts in the average case (defined suitably).
\end{enumerate}
\end{theorem}

\begin{proof}

We denote the protocol $EFRW$ (Envy Free Robertson-Webb).
It may call itself twice with some of the $A_i$'s shifted to the $B$-side,
and with part of the cake.

\noindent
{\bf PROTOCOL $EFRW(A_1,\ldots,A_n;B_1,\ldots,B_m;C;\epsilon)$}
\begin{enumerate}
\item
If $n=1$ then give $A_1$ the entire cake and the protocol is done.
\item
If $n=2$ then $A_1,A_2,B_1,\ldots,B_m$ 
run a near-exact $(m+2,2,\epsilon)$-protocol on
the cake to produce two pieces that $A_1$ thinks are identical and everyone
else thinks are within $\epsilon$ of $\frac{1}{2}$.
This takes $\omega$ cuts.
$A_2$ picks and keeps one of the pieces, $A_1$ keeps the other.
The protocol is done.
\item
(It must be that $n\ge 3$.)
$A_1,\ldots,A_n$;$B_1,\ldots,B_m$ 
run a near-exact-* $(n+m,n,\epsilon)$-protocol.
This takes $\omega$ cuts.
If $A_2,\ldots,A_n$ agree with $A_1$ that these pieces are all of size $\frac{1}{n}$,
then these pieces are given out (it does not matter how) and the protocol is done.
Else goto the next step.
\item
There is a piece $P$ that is controversial for 
$A_1,\ldots,A_n$.  
Let $\delta$ be a parameter to be picked later (it will depend on $\epsilon,n,m$).
$A_1,\ldots,A_n;B_1,\ldots,B_m$ run a
controversial $(n,m,\delta)$-protocol.
This step takes $\omega$ cuts.
\item
There is a piece $P$, numbers $\beta<\alpha\le \delta$ 
and (after renumbering) $1\le i\le n-1$ such that
\begin{itemize}
\item
$A_1,\ldots,A_i$ all think $P$ is worth $\alpha$
\item
$A_{i+1}\ldots,A_n$ all think $P$ is worth $\le \beta$.
\item
$A_1,\ldots,A_n$,$B_1,\ldots,B_m$ all think $P$ is worth $\le \delta$. 
\end{itemize}
\item
Let $Q=C-P$.
$A_1,\ldots,A_n$;$B_1,\ldots,B_m$ run an {\it unfair $(n+m,f_1,f_2,\epsilon)$-protocol} 
to split $Q$ into $Q_1$ and $Q_2$ with $f_1,f_2$ picked 
such that all think 
$Q_1$ is just a shade less than $i/n$ of $Q$
and
$Q_2$ is just a shade more than $(n-i)/n$ of $Q$.
That shade is a function of $n,m,\epsilon$ and $a-b$.
\item
Run $EFRW(A_1,\ldots,A_i;A_{i+1},\ldots,A_n,B_1,\ldots,B_m;Q_1\cup P;\epsilon')$ 
(note that $A_{i+1},\ldots,A_n$ are now on the $B$-side) 
where $\epsilon'$ will be discussed later.
$Q_1\cup P$ is divided into $i$ pieces and given to $A_1,\ldots,A_i$ in an envy free
manner, while $A_{i+1},\ldots,A_n,B_1,\ldots,B_m$
think each piece is within $\epsilon'$ of  $\frac{1}{i}$ of $Q_1 \cup P$.
\item
Run
$EFRW(A_{i+1},\ldots,A_n;A_1,\ldots,A_i,B_1,\ldots,B_m;Q_2;\epsilon')$ 
(note that $A_{1},\ldots,A_i$ are now
on the $B$-side).
$Q_2$ is divided into $n-i$ pieces and given to 
$A_{i+1},\ldots,A_n$ in an envy free manner while  $A_1,\ldots,A_i,B_1,\ldots,B_m$
think each piece is within $\epsilon'$ of $\frac{1}{n-i}$ of $Q_2$.
\end{enumerate}

We pick that {\it shade less than $i/n$} carefully: close enough
to $i/n$ so that $A_1,\ldots,A_i$ think that getting $Q_1\cup P$ is worth 
getting
a shade less than $i/n$, 
but big enough so that $A_{i+1},\ldots,A_n$ thinks that getting that 
shade is worth more than $P$.
Such a shade exists since $A_1,\ldots,A_i$ value $P$ more than $A_{i+1},\ldots,A_n$.
We pick $\epsilon'$ so small that 
(1) $A_1, \ldots,A_i$ 
do not mind that $A_{i+1},\ldots,A_n$ may get 
$\epsilon'$ more than $\frac{n-i}{n}$ of
$Q_2$, and
(2) 
$A_{i+1},\ldots,A_n$ 
do not mind that $A_{1},\ldots,A_i$ may get 
$\epsilon'$ more than $\frac{i}{n}$ of
$Q_1\cup P$.

What about the $B_i$s? 
The parameter $\delta$ and $\epsilon'$ are picked small enough so that
at the end the $B_i$s see $A_1,\ldots,A_n$
getting within $\epsilon$ of $\frac{1}{n}$.

Let $T(n;m)$ be the number of cuts this protocol takes.

$(\forall m\ge 0)[T(1;m)=0]$

$(\forall m\ge 0)[T(2;m)=\omega]$ 

If $n\ge 3$ and $m\ge 1$ then the protocol will take $2\omega$ cuts and
then recurse. Hence

$T(n;m)\le 2\omega + \max_{1\le i\le n-1} (T(i;n+m-i) + T(n-i;m+i) )$

One can easily show that $(\forall n\ge 1)(\forall m\ge 0)[T(n;m)\le (2n-3)\omega]$.
In particular $T(n;0)\le (2n-3)\omega$. 
Hence when used for $n$-player envy free cake cutting, this protocol takes $(2n-3)\omega$ cuts
in the worst case.

We can study the average case by assuming that the players partitioning is random.
This leads to an average case of $(2n-O(1))\omega$.

\end{proof}

Note that the protocol from Theorem~\ref{th:efrw}, yields a 4-person $5\omega$-cuts envy free protocol.
Hence, for the case of $n=4$, it uses (essentially) the same number of cuts as the protocol
from Theorem~\ref{th:efbt}.

\section{Pikhurto's Protocol}

For our purposes Pikhurto's protocol is similar to the Robertson-Webb protocol so
we discuss it briefly and informally. 

In the Robertson-Webb protocol the $A$-players are partitioned into {\it two} groups:
those who think $P$ is size $\alpha$ and those who think $P$ is of size 
$\le \beta$.
In Pikhurto's protocol the $A$-players are partitioned into many groups
and within a group the opinion of $P$ is the same. 
Then the protocol calls itself on each group.

Let $T(n;m)$ be the number of cuts this protocol takes.

$(\forall m\ge 0)[T(1;m)=0]$

$(\forall m\ge 0)[T(2;m)=\omega]$ 

If $n\ge 3$ and $m\ge 1$ then the protocol will take $2\omega$ cuts and
then recurse on each group. Hence

$
T(n;m)\le 
2\omega + 
\max_{ \{ i_1,\ldots,i_k \st i_1+\cdots+i_k=n\} } 
T(i_1;n+m-i_1) + \cdots + T(i_k;n+m-i_k)
$

One can easily show that $T(n;m)\le (2n-3)\omega$ and that if the partition of the players is random
then the average case is $(2n-O(1))\omega$.

\section{Open Problems}

Is there an $n$-player  $\omega+O(1)$-cut envy free protocol?
Can the results for small values of $n$ be improved from what we have here?
Sam Zbarsky has obtained (unpublished) a protocol for $n=4$ 
that takes only $2\omega +O(1)$ cuts, in constrast to what we obtained
which was $5\omega+O(1)$ cuts.
His approach is rather complicated and does not seem to generalize; however, it is a proof-of-concept
that special case algorithms may do better than those presented in our paper.

\section{Acknowledgments}

Omitted because of double-blind policy.



\end{document}